\documentclass{amsart}

\usepackage[latin1]{inputenc}
\usepackage{amsthm}
\usepackage{amsmath}
\usepackage{amssymb}
\usepackage{mathrsfs}
\usepackage{enumitem}

\usepackage[T1]{fontenc}

\usepackage{tikz}
\usetikzlibrary{matrix,arrows,decorations.pathmorphing}

\usepackage[hidelinks]{hyperref}

\theoremstyle{plain}
\newtheorem{theo}{Theorem}
\newtheorem{prop}[theo]{Proposition}
\newtheorem{lem}[theo]{Lemma}
\newtheorem{cor}[theo]{Corollary}

\theoremstyle{definition}

\newtheorem{defi}{Definition}

\theoremstyle{remark}
\newtheorem{Rk}{Remark}

\DeclareMathOperator{\dens}{dens}

\DeclareMathOperator{\re}{Re}

\newcommand{\diff}{\mathop{}\!\mathrm{d}}

\begin{document}
	
	\title[Weak inclusiveness of the distribution of primes in residue classes]{Limiting properties of the distribution\\ of primes in an arbitrarily large number of residue~classes}
\author{Lucile Devin}
\email{devin@crm.umontreal.ca}
\address{Centre de recherches mathématiques,
	Université de Montréal,
	Pavillon André-Aisenstadt,
	2920 Chemin de la tour,
	Montréal, Québec, H3T 1J4, Canada}

\keywords{Chebyshev's bias, almost-periodic functions, regularity of measures}
\subjclass[2010]{11K70, 28C15, 11M26}
\date\today

\begin{abstract}
	We generalize current known distribution results on Shanks--Rényi prime number races to the case where arbitrarily many residue classes are involved. 
	Our method handles both the classical case that goes back to Chebyshev and function field analogues developed in the recent years.
	More precisely, let~$\pi(x;q,a)$ be the number of primes up to~$x$ that are congruent to~$a$ modulo~$q$. For a fixed integer~$q$ and distinct invertible congruence classes~$a_0,a_1,\ldots,a_D$, 	assuming the generalized Riemann Hypothesis and a weak version of the linear independence hypothesis, we show that the set of real~$x$ for which the inequalities~$\pi(x;q,a_0)>\pi(x;q,a_1)> \ldots >\pi(x;q,a_D)$ are simultaneously satisfied admits a logarithmic density.
\end{abstract}

	\maketitle
	
	\section{Prime number races}
	
	In \cite{ChebLetter}, Chebyshev detected a bias in the distribution of prime numbers in arithmetic progressions. 
	His observation, under the name of prime number races or Chebyshev's bias, has been studied and generalized extensively these last 25 years, since Rubinstein and Sarnak published their influential paper~\cite{RS} (see \cite{FordKonyagin_expository,GranvilleMartin} for exposition of the subject and references).
	In particular, the aim is to study the asymptotic behavior of~$\pi(x; q,a)$; the number of primes~$\leq x$ that are congruent to~$a \bmod q$. For a fixed integer~$q$ and distinct invertible congruence classes~$a_0,a_1,\ldots,a_D \bmod q$, one investigates the  inequalities~$\pi(x;q,a_0)>\pi(x;q,a_1)> \ldots >\pi(x;q,a_D)$. We use the logarithmic density and the terminology of~\cite{MartinNg}.
	
	\begin{defi}\label{Def log density}
	Let~$\mathcal{P}$ be a set of positive real numbers, the \emph{logarithmic density}~$\delta(\mathcal{P})$ of~$\mathcal{P}$ is given by
	\begin{equation}\label{equation def log density}
	\delta(\mathcal{P}) = \lim_{X\rightarrow\infty}\frac{1}{\log X}\int_{1}^{X}\mathbf{1}_{\mathcal{P}} (x)\frac{\diff x}{x} = \lim_{Y\rightarrow\infty}\frac{1}{Y}\int_{0}^{Y}\mathbf{1}_{\mathcal{P}}(e^{y})\diff y
	\end{equation}
	if the limit exists, where~$\mathbf{1}_{\mathcal{P}}$ is the indicator function of the set~$\mathcal{P}$.
	
	For fixed~$q$ and distinct invertible congruence classes~$a_0,a_1,\ldots,a_D \bmod q$, if, for each permutation $\sigma$, the set~$\mathcal{P}_{q;a_{\sigma(0)},a_{\sigma(1)},\ldots,a_{\sigma(D)}} :=\lbrace x \geq 1 : \pi(x;q,a_{\sigma(0)})>\pi(x;q,a_{\sigma(1)})> \ldots >\pi(x;q,a_{\sigma(D)}) \rbrace$  admits a logarithmic density, we say that the prime number race is weakly inclusive.
	\end{defi}

Note that if the limit does not exist in~\eqref{equation def log density}, one can still define the inferior (resp. superior) logarithmic density $\underline{\delta}(\mathcal{P})$ (resp. $\overline{\delta}(\mathcal{P})$) using the lower (resp. upper) limit.
	
	Among other things, Rubinstein and Sarnak~\cite{RS} showed conditionally that all prime number races are weakly inclusive; and in particular they calculated $\delta(\mathcal{P}_{4;3,1})\approx 0.9959$.
	In the general case of a prime number race between congruence classes modulo $q$, this result is conditional on the Generalized Riemann Hypothesis~(GRH) and the Linear Independence~(LI) conjecture for the zeros of the Dirichlet~$L$-functions of the non-principal characters modulo~$q$; that is to say that the non-trivial zeros of these $L$-functions have real part equal to~$\frac{1}{2}$, and that the multi-set
	\begin{align}\label{eq set of zeros prime}
	\mathcal{Z}(q) := \bigcup_{\substack{\chi \bmod q \\ \chi\neq \chi_0}}\lbrace \gamma \geq 0 : L(\tfrac{1}{2} + i\gamma,\chi) = 0 \rbrace
	\end{align} is linearly independent over~$\mathbf{Q}$. 
	
	Later, Martin and Ng~\cite[Th. 1.11(a)]{MartinNg} obtained the weak inclusiveness under a weaker version of the Linear Independence hypothesis. Namely they proved that a prime number race between $D+1$ contestants modulo~$q$ is weakly inclusive if enough independent Dirichlet characters modulo~$q$ have $2D+1$ self-sufficient zeros (i.e. elements of $\lbrace \gamma \geq 0 : L(\tfrac{1}{2} + i\gamma,\chi) = 0 \rbrace$ that are linearly independent of all the other elements of~$\mathcal{Z}(q)$). This condition has been further weakened in~\cite{Devin_Chebyshev} in the case of a prime number race between two contestants~($D=1$). In this paper, we extend the weakening of LI for the weak inclusiveness of prime number races with any fixed number of contestants as a corollary of our main theorem (Corollary~\ref{Cor weakly inclusive}).	

The setting of prime number races has been generalized to the study of a wider range of error terms in prime number theory using almost periodic functions (see e.g. \cite{ANS}).
Applying the explicit formulas and classic conjectures, the error terms in asymptotic formulas for sums of coefficients of $L$-functions at prime numbers can indeed be expressed using an almost periodic-function in the logarithm of the variable. 
We state our main result in this general setting. We study the limiting behavior of an almost periodic function with a finite number of terms and we will see in the proof of Corollary~\ref{Cor weakly inclusive} how to handle the fact that those sum are in general infinite (see Remark~\ref{Rk how to generalize}).
We show that such functions are not locally constant in a broad sense: the values of an almost-periodic function do not stay a positive proportion of the time in a smaller subspace than the one generated by all its values.

\begin{theo}\label{Theorem main primes}
	Let $N\geq 1$ and $\gamma_1,\ldots,\gamma_N >0$ be distinct real numbers.
	Let $D\geq 1$ and
	fix $\mathbf{c}_{1},\ldots, \mathbf{c}_{N} \in \mathbf{C}^{D}$.
	Let $F:\mathbf{R}^{+} \rightarrow \mathbf{R}^{D}$ be the function defined by
	\begin{equation*}
	F(x) = \sum_{n=1}^{N}\left( \mathbf{c}_{n}x^{i\gamma_n} + \overline{\mathbf{c}_{n}}x^{-i\gamma_n}\right) = \sum_{n=1}^{N}\left( \mathbf{a}_{n}\cos(\gamma_n\log x) + \mathbf{b}_{n}\sin(\gamma_n \log x)\right).
	\end{equation*}
	The image of $F$ is contained in a compact subset of the subspace $V_F$ of $\mathbf{R}^D$ generated by the vectors $\mathbf{a}_1, \ldots, \mathbf{a}_{N}, \mathbf{b}_{1}, \ldots, \mathbf{b}_{N}$.

	Then, for every subspace $H\subset\mathbf{R}^{D}$ not containing $V_F$ and every vector $\alpha\in\mathbf{R}^D$, one has $\delta(F \in \alpha + H) = 0$.
	
	In particular, if $V_F\not\subset \bigcup_{d=1}^{D}\lbrace x \in \mathbf{R}^{D} : x_d = 0\rbrace$, then, for every $\alpha_1,\ldots,\alpha_D\in \mathbf{R}$, one has 
	$\delta(F_{d}=\alpha_d) = 0$ for all $1\leq d\leq D$,
	and the logarithmic density	$$\delta(F_{1}>\alpha_1,F_{2}>\alpha_2,\ldots,F_{D}>\alpha_D)$$ exists.	
\end{theo}

\begin{Rk}
	\begin{enumerate}
		\item Note that up to applying a linear transformation to the function~$F$, the set~$\lbrace F_{1}>\alpha_1,F_{2}>\alpha_2,\ldots,F_{D}>\alpha_D \rbrace$ represents any set defined by linear inequalities between the coordinates.
		\item One would rather want to show more generally that for any subset $A$ of Lebesgue measure zero in $V$, the set~$\lbrace x: F(x)\in A\rbrace$ has logarithmic density equal to zero\footnote{Note that under stronger conditions Martin and Ng \cite[Cor. 4.7]{MartinNg} show that the limiting distribution of $F$ is absolutely continuous with respect to the Lebesgue measure on $V$, in particular, it does not assign mass to subsets of $V$ of Lebesgue measure equal to zero.}, but this does not hold when the torus generated by the~$\gamma_n$'s has a smaller dimension than~$V$.
		\item Theorem~\ref{Theorem main primes} is a multi-dimensional version of \cite[Th. 2.2]{Devin_Chebyshev} (see also \cite[Th. 4]{KaczoRama} conditional only on GRH and giving at most two points where the density can be non-zero).
	\end{enumerate}
\end{Rk}

Applying Theorem~\ref{Theorem main primes} to the classical prime number race in congruence classes with $D+1$ contestants yields the following result on weak inclusiveness. 

\begin{cor}\label{Cor weakly inclusive}
	Let $q \geq 3$ be an integer, assume GRH for the Dirichlet~$L$-functions modulo $q$.
	Suppose there exists $M \geq 1$ and $\gamma_{1}, \ldots, \gamma_{M} \in \mathcal{Z}(q)$ relatively independent of $\mathcal{Z}(q) \smallsetminus \lbrace \gamma_{1},\ldots,\gamma_{M} \rbrace$, i.e.
	\[\langle\gamma_{1},\ldots,\gamma_{M}\rangle_{\mathbf{Q}} \cap \langle \mathcal{Z}(q) \smallsetminus \lbrace \gamma_{1},\ldots,\gamma_{M}  \rbrace \rangle_{\mathbf{Q}} = \lbrace 0 \rbrace,\]
	and such that for each non-principal character $\chi \bmod q$, there exists $1\leq m \leq M$ with $L(\frac{1}{2} + i\gamma_m,\chi) = 0$ but $L(\frac{1}{2} + i\gamma_m,\chi') \neq 0$ for $\chi'\neq\chi$.
	
	Then, every prime number race modulo $q$ is weakly inclusive and the ties have density zero.
	\end{cor}

Note that, although Corollary~\ref{Cor weakly inclusive} is conditional on the generalized Riemann Hypothesis, the condition needed on the zeros is definitely weaker than the Linear Independence hypothesis that is used in \cite{RS}. It is also weaker than the condition in \cite[Th. 1.10(a)]{MartinNg} since it allows rational linear relations between the $\gamma_n$'s. 
In particular, the second part of the assumption is automatic if the zeros are simple in the multi-set $\mathcal{Z}(q)$.
However we do not know if this condition is easier to prove than the Linear Independence hypothesis: it still involves knowledge on all the zeros. 
Such a condition can be studied and checked more easily on function fields since, in this setting, the $L$-functions have finitely many zeros. This idea is carried out in Section~\ref{Sec Poly} and we obtain analogues of Corollary~\ref{Cor weakly inclusive} that are somewhat stronger in function fields.

Since the counting functions $\pi(\cdot;q,a)$ are step functions with jumps at prime numbers, it could seem more natural to use a logarithmic density adapted to the set of primes instead of the one from Definition~\ref{Def log density} which is adapted to the set $[1,\infty)$. In \cite{LMP}, Lichtman Martin and Pomerance show that, in the case of a prime number race with two contestants, the two densities exist and are equal as long as the ties have density zero.
Combining Theorem~\ref{Theorem main primes} with \cite[Th. 2.1]{LMP}, we deduce a weak condition ensuring that the two densities exist simultaneously (and thus are equal).

\begin{cor}
	\label{Cor existence of density sum}
	Let $q \geq 3$ be an integer, assume GRH for the Dirichlet~$L$-functions modulo $q$.
Under the same conditions as in Corollary~\ref{Cor weakly inclusive}, we have that, 
	for any distinct invertible classes~$a_0,a_1,\ldots,a_D$~modulo~$q$,
	\begin{align*}
	\lim_{X\rightarrow\infty}\frac{1}{\log X} \sum_{\substack {p \in \mathcal{P}_{q;a_0,a_1,\ldots,a_D} \\ p\leq X}}\frac{\log p}{p} = \delta(\mathcal{P}_{q;a_0,a_1,\ldots,a_D}).
	\end{align*}
\end{cor}

The proof of Theorem~\ref{Theorem main primes} is based on the fact that the limit behavior of an almost-periodic function is well described by its limiting distribution. 
Let us first recall the definition.
\begin{defi}\label{Def limiting log distribution}
	Let $F:\mathbf{R}\rightarrow\mathbf{R}^{D}$ be a real function. 
	If it exists, the \emph{limiting logarithmic distribution} of the function $F$ is the  probability measure $\mu$ on $\mathbf{R}^{D}$ such that
	for any bounded Lipschitz continuous function $g$, we have
	\begin{align*}
	\lim_{X\rightarrow\infty}\frac{1}{\log X}\int_{1}^{X}g(F(x))\frac{\diff x}{x} = 
	\int_{\mathbf{R}^{D}}g\diff\mu.
	\end{align*}
\end{defi}

The existence of the limiting logarithmic distribution of an almost-periodic function $F$ with a finite number of terms, as in Theorem~\ref{Theorem main primes}, is a consequence of the Kronecker--Weyl Equidistribution Theorem as in \cite[Lem. 4.3]{Humphries} (see also \cite[Lem. 4.8]{Devin_Chebyshev}).

Thus, in general, studying a prime number race boils down to understanding the associated limiting logarithmic distribution.
We are mainly interested in the support and the regularity of this distribution.

If the distribution has a large enough support, for example intersecting all sectors, then one can prove that every contestant in the prime number race take the lead for a positive proportion of the time. This idea was developed in \cite[Th. 1.2]{RS} which is conditional only on the generalized Riemann Hypothesis only but works for a certain type of prime number races. In \cite[Rk.1.3]{RS} a more general result conditional on the generalized Riemann Hypothesis and the Linear Independence is stated, this result was improved in \cite[Th. 1.11(b)]{MartinNg} by weakening the hypothesis on the linear independence of the zeros. 

If the distribution does not have atoms, or does not concentrate mass on strict subspaces, then one can prove that ``the ties have density zero'', this also ensure the existence of logarithmic density (defined by the limit in \eqref{equation def log density}) for all set $\mathcal{P}$ defined by linear inequalities between the coordinates. In \cite{RS}, Rubinstein and Sarnak proved, under the Linear Independence hypothesis, that the distribution associated to any prime number race in congruence classes is smooth. Martin and Ng (\cite{MartinNg}) showed under a weaker assumption that this distribution is absolutely continuous. Both the results of Rubinstein--Sarnak and Martin--Ng imply that all measurable sets of Lebesgue measure\footnote{The Lebesgue measure we refer to in this case could be the Lebesgue measure on $\mathbf{R}^D$ or on a proper subspace, see the discussion in the end of \cite[Sec. 2.1]{MartinNg} for more details.} zero are also null sets for this distribution. The main point of Theorem~\ref{Theorem main primes} is that we give the weakest necessary condition to ensure that the distribution associated to a general prime number race with an arbitrary large number of contestants does not concentrate mass on strict subspaces of its support (without more regularity).


	\section{Chebyshev's bias in function fields}\label{Sec Poly}
	
	The analogy	between $\mathbf{Z}$ and rings of transcendence degree~$1$ over finite fields (see e.g. \cite{Rosen2002}) provides a natural translation of the notion of prime number races to irreducible polynomials with coefficients in finite fields.
	This translation was first studied by Cha in \cite{Cha2008} where he adapted the results of \cite{RS}. Further generalizations have been studied since then in \cite{ChaKim,ChaIm,CFJ,DevinMeng,Perret-Gentil}. 
	In this setting, some results are unconditional since the Riemann Hypothesis in function field of curves over finite fields is a theorem of Weil \cite{Weil_RH}.
	Moreover the hypothesis of linear independence is proven in generic cases \cite{Kowalski2010,CFJ_Indep,Perret-Gentil} but counterexamples can be found \cite{Cha2008,Kowalski2010,Li2018}.

In this context, instead of counting primes up to $x$ (a real number) we count irreducible monic polynomials of degree up to $k$ (an integer), it is natural to change the variable $x = q^k$ where $q$ is the cardinal of the field of coefficients. Thus we use natural densities instead of logarithmic densities, and we use a sum instead of an integral in the definition.
Note, however, that --- as observed in Corollary~\ref{Cor existence of density sum} --- the use of a sum or an integral in Definitions~\ref{Def log density} and~\ref{Def limiting log distribution} does not modify the results and values when the limiting logarithmic distribution does not give mass to hyperplanes.

\begin{defi}
	Let~$\mathcal{P}$ be a set of positive integers, the \emph{(natural) density}~$\dens(\mathcal{P})$ of~$\mathcal{P}$ is given by
	\begin{equation*}
	\dens(\mathcal{P}) = \lim_{X\rightarrow\infty}\frac{1}{X}\sum_{k\leq X}\mathbf{1}_{\mathcal{P}}(k) 
	\end{equation*}
	if the limit exists, where~$\mathbf{1}_{\mathcal{P}}$ is the indicator function of the set~$\mathcal{P}$.
\end{defi}

We will usually omit the qualifier ``natural'' in the rest of the paper, except when comparing with the logarithmic density.
Similarly to the logarithmic density, one can always define the inferior (resp. superior) density $\underline{\dens}$ (resp. $\overline{\dens}$). 

We have the following analogue of Theorem~\ref{Theorem main primes}.

\begin{theo}\label{Theorem main poly}
	Let $N\geq 1$ and $\gamma_1,\ldots,\gamma_N \in (0,\pi) \smallsetminus \mathbf{Q}\pi$ be distinct real numbers.
	Let $D\geq 1$ and
	fix $\mathbf{c}_{1},\ldots, \mathbf{c}_{N} \in \mathbf{C}^{D}$.
	Let $F:\mathbf{N} \rightarrow \mathbf{R}^{D}$ be the function defined by
	\begin{equation*}
	F(k) = \sum_{n=1}^{N}\left(\mathbf{c}_{n}e^{ik\gamma_n} + \overline{\mathbf{c}_{n}}e^{-ik\gamma_n}\right) = 
	\sum_{n=1}^{N}\left( \mathbf{a}_{n}\cos(\gamma_n k) + \mathbf{b}_{n}\sin(\gamma_n k)\right).
	\end{equation*}
	The image of $F$ is contained in a compact subset of the subspace $V_F$ of $\mathbf{R}^D$ generated by the vectors $\mathbf{a}_1, \ldots, \mathbf{a}_{N}, \mathbf{b}_{1}, \ldots, \mathbf{b}_{N}$. 

	Then, for every subspace $H\subset\mathbf{R}^{D}$ not containing $V_F$ and every vector $\alpha\in\mathbf{R}^D$, one has $\dens(F \in \alpha + H) = 0$.

In particular, if $V_F\not\subset \bigcup_{d=1}^{D}\lbrace x \in \mathbf{R}^{D} : x_d = 0\rbrace$, then, for every $\alpha_1,\ldots,\alpha_D\in \mathbf{R}$, one has 
$\dens(F_{d}=\alpha_d) = 0$ for all $1\leq d\leq D$,
and the density	$$\dens(F_{1}>\alpha_1,F_{2}>\alpha_2,\ldots,F_{D}>\alpha_D)$$ exists.	
\end{theo}

\begin{Rk}\begin{enumerate}
	\item Note that the difference between Theorem~\ref{Theorem main primes} and Theorem~\ref{Theorem main poly} is only on the set of possible values for the $\gamma_n$'s. When studying the limiting distribution of an almost-periodic function over the integers, the almost-periods commensurable to $\pi$ play a special role as $\lbrace k\pi : k\in\mathbf{Z}\rbrace/2\pi\mathbf{Z}$ is a discrete subgroup of $\mathbf{Z}/2\pi\mathbf{Z}$.
	This specificity does not appear in Theorem~\ref{Theorem main primes} since the set $\lbrace \pi\log x : x\in\mathbf{R}\rbrace/2\pi\mathbf{Z}$ is dense in $\mathbf{Z}/2\pi\mathbf{Z}$.
	\item It is expected --- and proved in some cases \cite{Kowalski2010,CFJ_Indep} --- that the Linear Independence hypothesis holds generically in families of $L$-functions. Thus it is natural to think that the condition of Theorem~\ref{Theorem main poly} is generically satisfied in various versions of Chebyshev's bias over function fields, and probably it holds more often than the Linear Independence hypothesis.
\end{enumerate}
\end{Rk}

As in the case of Theorem~\ref{Theorem main primes}, our results will in fact be about the regularity of the limiting distribution of our almost-periodic function.

\begin{defi}
	Let $F:\mathbf{N}\rightarrow\mathbf{R}^{D}$ be a real function. 
	If it exists, the \emph{limiting (natural) distribution} of the function $F$ is the probability measure $\mu$ on $\mathbf{R}^{D}$ such that
	for any bounded Lipschitz continuous function $g$, we have
	\begin{align*}
	\lim_{X\rightarrow\infty}\frac{1}{X}\sum_{k\leq X}g(F(k)) = 
	\int_{\mathbf{R}}g(t)\diff\mu(t).
	\end{align*}
\end{defi}

We obtain the analogue of Corollary~\ref{Cor weakly inclusive} in function fields, where similarly to \eqref{eq set of zeros prime}, we define the multi-set of ``non-negative'' zeros of $L$-functions modulo $Q\in\mathbf{F}_{p^{\alpha}}$, 
\begin{align*}
\mathcal{Z}(Q) := \bigcup_{\substack{\chi \bmod Q \\ \chi\neq \chi_0}}\lbrace \gamma \in [0,\pi] : L(\tfrac{1}{2} + i\gamma,\chi) = 0 \rbrace.
\end{align*}
	\begin{cor}\label{cor weakly inclusive poly}
	Let $p^{\alpha}$ be a power of prime, and $Q \in \mathbf{F}_{p^{\alpha}}[t]$. 
	If there exists $M \geq 1$ and $\gamma_{1}, \ldots, \gamma_{M} \in \mathcal{Z}(Q)\smallsetminus \mathbf{Q}\pi$ 
	such that, for each character $\chi \bmod Q$, there exists $1\leq m \leq M$ with $L(\frac{1}{2} + i\gamma_m,\chi) = 0$ but $L(\frac{1}{2} + i\gamma_m,\chi') \neq 0$ for $\chi'\neq\chi$.
	
	Then, every irreducible polynomial race in $\mathbf{F}_{p^{\alpha}}$ modulo $Q$ is weakly inclusive and the ties have density zero.
\end{cor}

The proof follows the same lines as the proof of Corollary~\ref{Cor weakly inclusive}. Note that this also applies to races between polynomials with a fixed number of irreducible factors in congruence classes as in \cite{DevinMeng} since the asymptotic formulas depend similarly on the zeros of the Dirichlet $L$-functions. 
 
In \cite{Kowalski2010}, Kowalski proved a generic Linear Independence for the zeros of Dirichlet~$L$-functions of quadratic characters in function fields. Together with Theorem~\ref{Theorem main poly}, this gives unconditionally the existence of Chebyshev's bias in many races between quadratic residues and non-quadratic residues.   
In particular, the following result is a direct consequence of Theorem~\ref{Theorem main poly} in the case $D=1$ with \cite[Prop. 1.1]{Kowalski2010} and \cite[(5)]{DevinMeng}. 

\begin{cor}\label{Cor generic existence of the bias}
	Let $f \in \mathbf{Z}[T]$ be a squarefree monic polynomial of degree $2g$, where
	$g \geq 1$ is an integer. Let $p$ be an odd prime such that $p$ does not divide the discriminant
	of $f$, and let $U/\mathbf{F}_p$ be the open subset of the affine line, defined by $f(u)\neq 0$. 
	For any extension~$\mathbf{F}_q /\mathbf{F}_p$, any~$u\in U(\mathbf{F}_q)$, and any subset~$A$ of the invertible classes in $\mathbf{F}_q[T]/(f(T)(T-u))$ we denote
	\begin{equation*}
	\Pi_{k}(x;u,A) :=\frac{1}{\lvert A\rvert}\left\lvert \left\lbrace N \in \mathbf{F}_{q}[T] \text{ monic: } \substack{\deg{N} \leq x, \\ \Omega(N)=k }, N \bmod f(T)(T-u) \in A \right\rbrace\right\rvert.
	\end{equation*}

Then, denoting $\square$ (resp. $\boxtimes$) the subset of quadratic residues (resp. non-quadratic residues), we have
\begin{align*}
\lvert \lbrace u \in U(\mathbf{F}_q): \dens(\Pi_{k}(\cdot;u,\square) > \Pi_{k}(\cdot;u,\boxtimes)) \text{ does not exists } \rbrace\rvert \ll_{g} q^{1 - \frac{1}{4g^2 + 2g +4}}\log q.
\end{align*}

\end{cor}

\begin{Rk}
Kowalski shows in fact that the Linear Independence is not satisfied for the Dirichlet~$L$-function of the primitive quadratic character modulo~$f(T)(T-u)$ for no more than this small portion of the $u \in U(\mathbf{F}_q)$. However, the polynomial~$f(T)(T-u)$ is not irreducible, we do not know if all the zeros of the  $L$-functions of the quadratic characters modulo $f(T)(T-u)$ are indeed linearly independent over $\mathbf{Q}$ for most of the $u$'s, so we cannot derive stronger properties of the limiting distribution (such as the symmetry or more smoothness that are obtained under various weaker versions of LI in \cite[Sec. 2.2]{Devin_Chebyshev}). To obtain Corollary~\ref{Cor generic existence of the bias} from \cite[Prop. 1.1]{Kowalski2010}, we only use the fact that if $L(s,\chi)$ satisfies LI, then at least one imaginary part of a zero of $L(s,\chi)$ is not in~$\mathbf{Q}\pi$, and the associated coefficient is non-zero thanks to the formula \cite[(5)]{DevinMeng}.	
\end{Rk}

Other cases of generic linear independence are obtained following the idea of \cite{Kowalski2010}.
Notably, in \cite{CFJ_Indep}, Cha, Fiorilli and Jouve prove that generically in certain families of elliptic curves over function fields, the hypothesis LI is satisfied for the $L$-functions of the elliptic curve. 
They use this result in \cite{CFJ} to obtain, among other things, an improved analogue of Corollary~\ref{Cor generic existence of the bias} in the context of Mazur's version of Chebyshev's bias for the coefficients $a_p$ of elliptic curves over function fields.

Similarly, in \cite{Perret-Gentil}, Perret-Gentil proves several results on generic linear independence, related to the function field version of the bias in the distribution of angles of Gaussian primes in sectors (see also \cite{RudnickWaxman}) and to the distribution of irreducible polynomials in short intervals (see also \cite{KeatingRudnick}).
However these results show the existence of small subsets of the zeros satisfying the linear independence, and this is not enough to apply Theorem~\ref{Theorem main poly} in the case $D \geq 2$ : we would need to control the multiplicities on the whole set of zeros to conclude.

\section{Proofs of the main results}
		
		Let us first prove Corollary~\ref{Cor weakly inclusive} assuming Theorem~\ref{Theorem main primes}. Note that the proof of Corollary~\ref{cor weakly inclusive poly} assuming Theorem~\ref{Theorem main poly} follows the same lines.
		
		\begin{proof}[Proof of Corollary~\ref{Cor weakly inclusive} assuming Theorem~\ref{Theorem main primes}]
			The proof follows the idea of the proof of \cite[Th. 1.10]{MartinNg}.
			Let $1\leq D \leq \phi(q) - 1$ and $a_0, a_1,\ldots a_D$ be distinct invertible classes modulo~$q$.
			We study the function~$E : \mathbf{R} \rightarrow \mathbf{R}^{D}$ with coordinates
			\begin{equation}\label{eq prime number race}
			E_d(x) = \frac{\log x}{\sqrt{x}}(\pi(x;q,a_d) - \pi(x;q,a_{d-1})), 
			\end{equation}
			for $1\leq d\leq D$.
			Using the explicit formula as in \cite[2.1]{RS} we get that, for each~$1\leq d\leq D$,
			\begin{align*}
			E_d(x) = 2\re\left(\sum_{m=1}^{M}(\bar{\chi}_{m}(a_{d}) - \bar{\chi}_{m}(a_{d-1})) \frac{x^{i\gamma_m}}{\frac{1}{2} + i\gamma_m} \right) + G_{d}(x)
			\end{align*}
			where for each~$1\leq m\leq M$, $\chi_m$ is defined uniquely such that one has~$L(\chi_m,\frac{1}{2} + i\gamma_m) = 0$ (recall that~$\mathcal{Z}(q)$ is a multi-set), and the function~$G = (G_1,\ldots,G_d)$ admits a limiting logarithmic distribution~$\mu_G$. Note also that by hypothesis the remaining zeros in $\mathcal{Z}(q)$ involved in~$G$ are linearly independent of~$\lbrace \gamma_1,\ldots, \gamma_M \rbrace$.
			
			Up to gathering the contributions of the zeros with multiplicities, The function~$E-G$
			has the form of the function~$F$ in Theorem~\ref{Theorem main primes}.
			Thus it admits a limiting logarithmic distribution~$\mu_F$, and the limiting logarithmic distribution of~$E$ is the convolution of the measures $\mu_E = \mu_F \ast \mu_G$ (this follows from the relative linear independence, similarly to \cite[Lem. B.4]{MartinNg} using \cite[Prop. 5.1]{Devin_Chebyshev}). 
			
			By orthogonality of the Dirichlet characters, the set~$\lbrace (\chi(a_1) - \chi(a_0),\ldots,\chi(a_D)-\chi(a_{D-1})) : \chi\bmod q \rbrace$ spans~$\mathbf{C}^{D}$. 
			Thus, applying Theorem~\ref{Theorem main primes} yields that the measure~$\mu_F$ does not assign mass to affine subspaces of $\mathbf{R}^{D}$.
			This property is transferred to~$\mu_E$ by convolution, and this gives the weak inclusiveness by Proposition~\ref{Prop existence of the bias}.
		\end{proof}
	
	\begin{Rk}\label{Rk how to generalize}
		Note that the proof is not specific to the case of Dirichlet characters, in general, under classical hypotheses, normalized error functions $E$ studied in prime number races context (see e.g. \cite[Th. 2.1]{Devin_Chebyshev}) are almost-periodic. If we can separate the almost-periods into two relatively independent subsets, one of them being finite, then we can deduce results on the regularity of $\mu_{E}$ from the regularity of $\mu_F$ where $F$ is as in Theorem~\ref{Theorem main primes}.
		
		In the setting of function fields, the separation is easier since the set of almost-periods is already finite. 
	\end{Rk}

Let us now give the proof of Corollary~\ref{Cor existence of density sum}, we use the same notations as in the proof of Corollary~\ref{Cor weakly inclusive}.			
		\begin{proof}[Proof of Corollary~\ref{Cor existence of density sum} assuming Theorem~\ref{Theorem main primes}]
			The result follows from~\cite[Th. 2.1]{LMP} adapted to the case of a function with values in $\mathbf{R}^D$ and applied to the function~$E$ defined in~\eqref{eq prime number race}. 
			The only difference with the one-dimensional case is that we now need to prove that one has
			\begin{align*}
			\lim_{\substack{\eta \rightarrow 0 \\ \eta \neq 0}}\mu_{E}(\lbrace x \in \mathbf{R}^D : x_{d}>\alpha_d + \eta, 1\leq d \leq D\rbrace) =
			\mu_{E}(\lbrace x \in \mathbf{R}^D : x_{d}>\alpha_d, 1\leq d \leq D\rbrace).
			\end{align*}	
			This is also a consequence of the fact that~$\mu_{E}$ does not assign mass to affine hyperplanes, associated with the Dominated Convergence Theorem. 
		\end{proof}

	The rest of the section is dedicated to the proof of Theorems~\ref{Theorem main primes} and~\ref{Theorem main poly}. 
	Note first that by change of variables, the limiting logarithmic distribution of the function $F$ is the limiting (natural) distribution of the function $F\circ\exp$.
	Thus the proofs of Theorem~\ref{Theorem main primes} and Theorem~\ref{Theorem main poly} follow the sames lines. We give below the proof of Theorem~\ref{Theorem main poly}, noting along the way the differences and adaptations for Theorem~\ref{Theorem main primes}.
	
	Recall that the existence of the limiting distribution $\mu$ of $F$ is a consequence of the Kronecker--Weyl Equidistribution Theorem as seen in \cite{MartinNg, Devin_Chebyshev}.
	In particular, the measure~$\mu$ is the push-forward of the Haar measure $\omega_{A}$ on the sub-torus 
	$$A = \overline{\lbrace (k\gamma_1, \ldots, k\gamma_N) : k\in \mathbf{Z}  \rbrace / (2\pi\mathbf{Z})^{N}}$$ 
	via the function
	$\tilde{F}: A \subset \mathbf{T}^{N} \rightarrow \mathbf{R}^{D}$ defined by
	\begin{equation*}
	\tilde{F}(\theta) =  \sum_{n=1}^{N} 2 \re\left( \mathbf{c}_{n}e^{i\theta_n} \right).
	\end{equation*}
	Note that in the case of Theorem~\ref{Theorem main primes}, the subtorus $A$ is parameterized by $\mathbf{R}$ instead; $$A_{\mathbf{R}} = \overline{\lbrace (y\gamma_1, \ldots, y\gamma_N) : y\in \mathbf{R}  \rbrace /(2\pi \mathbf{Z})^{N}}.$$ 
	In each of these theorems the second part on the existence of the (logarithmic or natural) density is a consequence of the first part.
	
	\begin{prop}\label{Prop existence of the bias}
		Let~$F :\mathbf{N} \rightarrow \mathbf{R}^D$, admitting a natural limiting distribution~$\mu$.
		Assume that~$\mu(\lbrace x \in \mathbf{R}^D : x_{d} =\alpha_d\rbrace) = 0$ for each~$1\leq d \leq D$. 
		
		Then $\dens(F_{d}=\alpha_d) = 0$ for each $1\leq d\leq D$, and the density  
		$\dens(F_{1}>\alpha_1,\ldots,F_{D}>\alpha_D)$ exists and equals~$\mu(\lbrace x \in \mathbf{R}^D : x_{1}>\alpha_1,\ldots,x_{D}>\alpha_D\rbrace)$
	\end{prop}

\begin{proof}
	Using continuous approximations of the indicator function, one gets that, for any~$\epsilon>0$,
	\begin{equation*}
	0
	\leq  \underline{\dens}(F_d = \alpha_d) 
	\leq \overline{\dens}(F_d = \alpha_d)  \leq \mu(\lbrace x \in \mathbf{R}^D :\alpha_d - \epsilon \leq x_{d}\leq\alpha_d + \epsilon\rbrace)
	\end{equation*}
 for each $1\leq d\leq D$,	and
	\begin{multline*}
	\mu(\lbrace x \in \mathbf{R}^D : x_{1}>\alpha_1 + \epsilon,\ldots,x_{D}>\alpha_D + \epsilon\rbrace) \\
	\leq  \underline{\dens}(F_{1}>\alpha_1,\ldots,F_{D}>\alpha_D) 
	\leq \overline{\dens}(F_{1}>\alpha_1,\ldots,F_{D}>\alpha_D)  \\\leq \mu(\lbrace x \in \mathbf{R}^D : x_{1}>\alpha_1 - \epsilon,\ldots,x_{D}>\alpha_D - \epsilon\rbrace).
	\end{multline*}
	Moreover
	\begin{multline*}
	\lvert \mu(\lbrace x \in \mathbf{R}^D : x_{1}>\alpha_1 - \epsilon,\ldots,x_{D}>\alpha_D - \epsilon\rbrace) - \mu(\lbrace x \in \mathbf{R}^D : x_{1}>\alpha_1 + \epsilon,\ldots,x_{D}>\alpha_D + \epsilon\rbrace) \rvert \\
	\leq \sum_{d=1}^{D} \mu(\lbrace x \in \mathbf{R}^D :\alpha_d - \epsilon \leq x_{d}\leq\alpha_d + \epsilon\rbrace) \\
	\longrightarrow_{\epsilon\rightarrow 0} \sum_{d=1}^{D} \mu(\lbrace x \in \mathbf{R}^D : x_{d} =\alpha_d\rbrace) =0
	\end{multline*}
	by Dominated Convergence Theorem.
	This concludes the proof.
\end{proof}

\begin{Rk}
	The proof is similar when~$\mu$ is the limiting logarithmic distribution and~$\dens$ is replaced by~$\delta$, giving the analogue of Proposition~\ref{Prop existence of the bias} adapted to Theorem~\ref{Theorem main primes}.
\end{Rk}

 Theorems~\ref{Theorem main primes} and~\ref{Theorem main poly} are a consequence of the following general result.

\begin{prop}\label{Prop zero mass}
	Let $\mu$ be a probability measure on~$\mathbf{R}^{D}$, supported in the subspace~$V\subset \mathbf{R}^{D}$.
	Assume there exists~$\epsilon >0$ such that for all~$\xi \in V$, one has 	
	$$\lvert \hat{\mu}(\xi) \rvert \ll \min(1,\lVert \xi \rVert^{-\epsilon}).$$
	Then~$\mu$ does not assign mass to strict affine subspaces of~$V$.
\end{prop}

We keep the proof of this proposition for later, and we focus on deducing  Theorems~\ref{Theorem main primes} and~\ref{Theorem main poly}.	
We begin with proving that the measure defined in Theorem~\ref{Theorem main poly} satisfies the hypothesis of Proposition~\ref{Prop zero mass}.
In the case~$D=1$ already stated in~\cite[Th. 2.2]{Devin_Chebyshev}, the key point is that the one-dimensional function~$\tilde{F}$ is not locally constant on~$A$. We generalize this property to a multi-dimensional function, replacing points by affine hyperplanes. 
The image of~$\tilde{F}$ is contained in the subspace~$V$. We show that there is no non-empty open subset of $A$ whose image by $\tilde{F}$ is contained in a strict affine subspace of~$V$. 
In the following we denote by~$\mathbf{S}$ the unit sphere of~$\mathbf{R}^D$.  

\begin{lem}\label{Lemma Not loc constant}
	Under the conditions of Theorem~\ref{Theorem main poly}, for every unit vector $\mathbf{u} \in V\cap\mathbf{S}$, the function $\langle \mathbf{u}, \tilde{F}\rangle$ is not locally constant on $A$.
	
	More precisely, there exist $K\geq 1$, $\eta >0$ and $B>0$ such that 
	for all~$(\mathbf{u},\theta) \in (V\cap\mathbf{S})\times A$, 
 \begin{equation*}
 \eta \leq  \max_{1\leq \lvert \mathbf{k}\rvert \leq K}{\left\lvert  \frac{\partial^{\mathbf{k}}}{\partial \theta^{\mathbf{k}}}\langle \mathbf{u}, \tilde{F}(\theta)\rangle \right\rvert} \leq B,
 \end{equation*}
 where $\lvert \mathbf{k}\rvert  = k_1 + \ldots + k_{\dim A}$ is the length of the multi-index $\mathbf{k} \in \mathbf{N}^{\dim A}$.
\end{lem}

\begin{Rk}
	Note that when we write $\theta = (\theta_1,\ldots,\theta_N) \in A$ we need to keep in mind that this object is in fact in a torus of dimension $\dim A \leq N$ (the coordinates can be linearly related). In particular we use only $\dim A$ directions of differentiation for the variable $\theta$. 
\end{Rk}

\begin{proof}
	Let $f : (V\cap\mathbf{S}) \times A$ be the function defined by 
	\begin{align}\label{formula def aux function}
	f(\mathbf{u},\theta) :=& \langle \mathbf{u}, \tilde{F}(\theta)\rangle = \sum_{n=1}^{N}
	2\re( \langle \mathbf{u}, \mathbf{c}_n\rangle e^{i\theta_n}) 
	\\ =& \sum_{n=1}^{N}\big( \langle \mathbf{u}, \mathbf{a}_n\rangle \cos\theta_n + \langle \mathbf{u}, \mathbf{b}_n\rangle \sin\theta_n \big) \nonumber
	\end{align}
	 this function is analytic in the variable $\theta$. 
	
	Since the vectors $\mathbf{a}_{1}, \ldots, \mathbf{a}_{N}, \mathbf{b}_{1}, \ldots, \mathbf{b}_{N}$ span $V$, there exists at least one of them, say $\mathbf{v}_{j}$, such that $\langle \mathbf{u}, \mathbf{v}_{j} \rangle\neq 0$. Thus by a result of Dedekind--Artin (\cite[VI, Th. 4.1]{LangAlg} as used in \cite[Lem. 5.5]{Devin_Chebyshev}), the function $f(\mathbf{u},\cdot)$ is not constant on $A$.
	Note that the fact that the $\gamma_n$'s are all distinct ensure that the $\theta_n$'s are all distinct which allows to apply the linear independence of the functions $\cos\theta_n$, $\sin\theta_n$. 
	Moreover the fact that $\gamma_{j}\notin\mathbf{Q}\pi$ ensures that $A$ is not discrete in the direction of $\theta_{j}$ so that a non-constant analytic function in this variable is indeed non-locally constant. 
	
	This means that for~$\mathbf{u}\in V\cap\mathbf{S}$ fixed, around each~$\theta \in A$, there exist a neighborhood, on which at least one partial derivative of $f(\mathbf{u},\cdot)$ does not vanish. 
	Moreover, the set $(V\cap\mathbf{S}) \times A$ is compact, and for every~$\mathbf{k} \in \mathbf{N}^{\dim A}$ the function~$\frac{\partial^{\mathbf{k}}}{\partial \theta^{\mathbf{k}}}f(\mathbf{u},\theta)$ is continuous in the variables~$\mathbf{u}$ and~$\theta$.
	This gives the uniformity.
\end{proof}

\begin{Rk}
	The analogue of Lemma~\ref{Lemma Not loc constant} in the context of Theorem~\ref{Theorem main primes} follows from the same lines.
	Note that the fact that $A_{\mathbf{R}}$ is not discrete is obtained unconditionally in this case since $\mathbf{R}$ is not discrete.
\end{Rk}

We now use Lemma~\ref{Lemma Not loc constant} to obtain a decay of the Fourier Transform of $\mu$.

\begin{lem}\label{Lemma epsilon}
	Under the conditions of Theorem~\ref{Theorem main primes} or of Theorem~\ref{Theorem main poly}, and with $K$ satisfying Lemma~\ref{Lemma Not loc constant},
	 one has, for all~$\xi \in V$,
	$$\lvert \hat{\mu}(\xi) \rvert \ll \min(1,\lVert \xi \rVert^{-1/K}).$$
\end{lem}

\begin{proof}
	Let us write $\xi = \lVert \xi \rVert \mathbf{u}$, with $\mathbf{u} \in V\cap\mathbf{S}$ a unit vector.
	Then
	\begin{align*}
	\hat{\mu}(\xi) &= \int_{A}\exp(-2\pi i \langle\xi,\tilde{F}(\theta)\rangle) \diff\omega_{A}(\theta) \\
	&= \int_{A}\exp\left(-4\pi i \lVert \xi \rVert 
	f(\mathbf{u},\theta)\right) \diff\omega_{A}(\theta),
	\end{align*} 
	where $f$ is as defined in~\eqref{formula def aux function}. 
	By Lemma~\ref{Lemma Not loc constant} and \cite[VIII 2.2 Prop. 5]{Stein}, one has
	\begin{align*}
	\lvert \hat{\mu}(\xi) \rvert \ll  C(f(\mathbf{u},\cdot))\lVert \xi \rVert^{-1/K}, 
	\end{align*}
	where the implicit constant depends on $\eta, K$ and on a choice of partition of unity adapted to the finite open cover of $(V\cap\mathbf{S}) \times A$ described implicitly in the proof of Lemma~\ref{Lemma Not loc constant} (in particular the implicit constant does not depend on $\mathbf{u}$),
	and the ``constant''~$C(f(\mathbf{u},\cdot))$ remains bounded as long as the norm $\max_{\lvert \mathbf{k}\rvert \leq K}{\lVert  \frac{\partial^{\mathbf{k}}}{\partial \theta^{\mathbf{k}}}f(\mathbf{u},\cdot) \rVert_{\infty}}$ is bounded (which follows also from Lemma~\ref{Lemma Not loc constant}).
	
	We conclude the proof of Lemma~\ref{Lemma epsilon} by noting that the trivial bound $\lvert\hat{\mu}(\xi)\rvert \leq 1$ follows from bounding the exponential in the integral by $1$.
\end{proof}

Finally, we give the proof of Proposition~\ref{Prop zero mass} to conclude the proof of  Theorems~\ref{Theorem main primes} and~\ref{Theorem main poly}.

\begin{proof}[Proof of Proposition~\ref{Prop zero mass}]
	We use a continuous function approximating the indicator function of an hyperplane.
	Let $\alpha + H\subset \mathbf{R}^{D}$ be a strict affine subspace of $V$. Without loss of generality\footnote{Note that the case $\alpha = \mathbf{0}$ can be treated similarly using the convention that $\frac{\alpha}{\lVert \alpha\rVert}$ is a normal vector of $H$.}, we take $\mathbf{0}\neq \alpha \in V$ proportional to a normal vector of $H$, and we can extend $H$ to an hyperplane in $\mathbf{R}^{D}$ not containing $V$ (namely, the orthogonal of $\alpha$). 
	
	Then every element $x\in\mathbf{R}^D$ can be written as $x = h + t\frac{\alpha}{\lVert \alpha\rVert}$, with $h\in H$ and $t\in \mathbf{R}$.
	We define a continuous approximation of the indicator function of $\alpha + H$ on $\mathbf{R}^{D}$, for all $n\geq 1$:
	\begin{align}\label{Approximation of characteristic}
	g_{\alpha+H,n}\left(h + \frac{t}{\lVert \alpha\rVert}\alpha\right) = \begin{cases}
	0 \qquad &\text{ if } \lvert t - \lVert \alpha\rVert\rvert >\frac{1}{n} \\
	nt -n\lVert \alpha\rVert+1 \qquad &\text{ if }  \lVert \alpha\rVert -\frac{1}{n} \leq t\leq \lVert \alpha\rVert\\
	-nt +n\lVert \alpha \rVert + 1 \qquad &\text{ if } \lVert \alpha\rVert\leq t \leq \lVert \alpha\rVert+\frac{1}{n}. \\
	\end{cases} 
	\end{align}
	In particular, $\mathbf{1}_{\alpha + H} \leq g_{\alpha+H,n}$ for all $n \geq 1$.

	Using Plancherel's theorem we have:
	\begin{align*}
	\mu(\alpha + H) \leq \int_{\mathbf{R}^{D}} g_{\alpha+H,n} \diff\mu &=
	\int_{\mathbf{R}^{D}} \hat{g}_{\alpha+H,n}(\xi) \hat{\mu}(\xi) \diff\xi \\
	&=\int_{\xi\in H^{\bot}} e^{-2\pi i\langle \alpha,\xi\rangle} 4n\left(\frac{\sin(\lVert\xi\rVert/2n)}{\lVert\xi\rVert}\right)^{2}\hat{\mu}(\xi)\diff\xi, 
	\end{align*}
	where we used $H^\bot = \mathbf{R}\alpha \subset V$ and the Fourier transform of the constant function on $H$. Using the bound on the Fourier transform of the measure~$\mu$ on $V$, we obtain
	\begin{align*}
	\left\lvert\int_{\mathbf{R}^{D}} g_{\alpha+H,n} \diff\mu \right\rvert \ll  
	\int_{\mathbf{R}} 4n\min\left(\frac{1}{4n^2},\frac{1}{t^{2}}\right)\min(1,\lvert t \rvert^{-\epsilon})\diff t.
	\end{align*}
	Cutting the integral at $\lvert t \rvert = n^{1-\epsilon/2}$ and selecting the minimal value in each interval yields
	\begin{align*}
	\mu(\alpha + H) \leq \left\lvert\int_{\mathbf{R}^{D}} g_{\alpha+H,n} \diff\mu \right\rvert \ll  
n^{-\epsilon/2} + n^{-\epsilon(1-\epsilon)/2},
	\end{align*}
for all $n\geq 1$. This concludes the proof in the case $\epsilon<1$, for the case $\epsilon\geq 1$ one can cut the integral at $\lvert t \rvert = n^{2/3}$ and obtain $\mu(\alpha + H) \ll n^{-1/3}$ for all $n$.
\end{proof}

\section*{Acknowledgments}

	The results presented in this paper were motivated by questions of Greg Martin and Nathan Ng, and of Corentin Perret-Gentil. I thank them for encouraging me to write it and for giving me useful feedback.
	I also thank Daniel Fiorilli, Andrew Granville, and Florent Jouve for their advice and comments on this project.
	The paper also benefited from comments by Alexandre Lartaux.
	This work was supported by a Postdoctoral Fellowship at the University of Ottawa.
	
	\bibliographystyle{amsalpha}
	\bibliography{biblio}

\end{document}